\documentclass[a4paper,10pt]{article}
\usepackage{amssymb,amsmath,amsthm, amsbsy }
\usepackage{hyperref}
\usepackage{graphicx}
\usepackage{eucal}
\usepackage{color}
\newcommand{\ie}{\emph{i.e.}}
\newcommand{\eg}{\emph{e.g.}}
\newcommand{\cf}{\emph{cf.}}
\newcommand{\Real}{\mathbb{R}}
\newcommand{\Nat}{\mathbb{N}}

\newcommand{\dist}{\mathop{\mathrm{dist}}\nolimits}
\newcommand{\Dom}{\mathfrak{D}}
\newcommand{\eps}{\varepsilon}
\newcommand{\sii}{L^2}
\newtheorem{Theorem}{Theorem}
\newtheorem{Proposition}{Proposition}
\newtheorem{Corollary}{Corollary}
\newtheorem{Conjecture}{Conjecture}
\newtheorem{Lemma}{Lemma}

\usepackage[normalem]{ulem}
\definecolor{DarkGreen}{rgb}{0,0.5,0.1}

\newcommand\soutD{\bgroup\markoverwith
{\textcolor{DarkGreen}{\rule[.5ex]{2pt}{1pt}}}\ULon}
\newcommand{\Hm}[1]{\leavevmode{\marginpar{\tiny%
$\hbox to 0mm{\hspace*{-0.5mm}$\leftarrow$\hss}%
\vcenter{\vrule depth 0.1mm height 0.1mm width \the\marginparwidth}%
\hbox to 0mm{\hss$\rightarrow$\hspace*{-0.5mm}}$\\\relax\raggedright
#1}}}

\begin{document}
%
\title{\textbf{\Large Hardy inequalities in globally twisted waveguides}}
\author{Philippe Briet$^{a}$,
Hiba Hammedi$^{a}$
and David Krej\v{c}i\v{r}\'{\i}k$^{b}$}
\date{\small
\emph{
\begin{quote}
\begin{itemize}
\item[$a)$]
Aix-Marseille Universit\' e, CNRS, CPT, UMR 7332,
Case 907, 13288 Marseille \&
Universit\' e de Toulon, CNRS, CPT, UMR 7332
83957, La Garde, France;
briet@cpt.univ-mrs.fr, hammedi@cpt.univ-mrs.fr.%
\\
\item[$b)$]
Department of Theoretical Physics, Nuclear Physics Institute ASCR,
25068 \v{R}e\v{z}, Czech Republic;
krejcirik@ujf.cas.cz.%
\end{itemize}
\end{quote}
}
\medskip
11 June 2014 }

\maketitle

\begin{abstract}
\noindent
We establish various Hardy-type inequalities
for the Dirichlet Laplacian
in perturbed periodically twisted
tubes of non-circular cross-sections.
We also state conjectures about the existence
of such inequalities in more general regimes,
which we support by heuristic and numerical arguments.
%
%
\end{abstract}
%

%
%
\section{Introduction}\label{Sec.Intro}
\emph{``I have never done anything `useful'.
No discovery of mine has made, or is likely to make,
directly or indirectly, for good or ill,
the least difference to the amenity of the world.''}
This is a quotation from a 1940 essay~\cite{Hardy-apology}
by British mathematician G.~H.~Hardy.
Despite of this self-identification as a pure mathematician,
his work has found important applications in physics,
including the celebrated \emph{Hardy inequality}~\cite{Hardy_1920}
\begin{equation}\label{Hardy.1D}
  \forall \varphi\in H_0^1((0,\infty))
  \,, \qquad
  \int_0^\infty |\varphi'(x)|^2 \, dx
  \geq \frac{1}{4} \int_0^\infty \frac{|\varphi(x)|^2}{|x|^2} \, dx
  \,.
\end{equation}
For instance, using~\eqref{Hardy.1D} in the radial component
of the three-dimensional Laplacian, the inequality directly explains
the \emph{stability} of hydrogen-type atoms in quantum mechanics.

In a different quantum-mechanical context,
Ekholm, Kova\v{r}\'ik and the last author
employed~\eqref{Hardy.1D} to prove in~\cite{EKK}
certain stability of the spectrum of the Dirichlet Laplacian
in \emph{locally twisted tubes}
\begin{equation}\label{tube}
  \Omega :=
  \left\{
  \left.
  \begin{pmatrix}
    1 & 0 & 0 \\
    0 & \cos\theta(x_1) & \sin\theta(x_1) \\
    0 & -\sin\theta(x_1) & \cos\theta(x_1)
  \end{pmatrix}
  \begin{pmatrix}
    x_1 \\ x_2 \\ x_3
  \end{pmatrix}
  \ \right| \
  (x_1,x_2,x_3) \in \Real \times \omega
  \right\}
  .
\end{equation}
Here the cross-section~$\omega$ is
an arbitrary bounded open connected set in~$\Real^2$
and $\theta:\Real\to\Real$ is the twisting angle.
Assuming that~$\dot\theta$ is a compactly supported continuous function
with bounded derivative, the authors of~\cite{EKK} derived
a waveguide-type analogue of~\eqref{Hardy.1D}, namely,
\begin{equation}\label{Hardy.tube}
  \forall \psi\in H_0^1(\Omega)
  \,, \qquad
  \int_\Omega |\nabla\psi(x)|^2 \, dx
  - E_1 \int_\Omega |\psi(x)|^2 \, dx
  \geq c_H \int_\Omega \frac{|\psi(x)|^2}{1+|x|^2} \, dx
  \,.
\end{equation}
Here~$E_1$ denotes the first eigenvalue of
the Dirichlet Laplacian in~$\omega$, $-\Delta_D^\omega$.
The constant~$c_H$ is positive if, and only if,
$\dot\theta$~is not identically zero and~$\omega$ is not rotationally
invariant with respect to the origin in~$\Real^2$.
If~$\Omega$ is twisted \emph{locally} 
in the sense that~$\dot\theta$ vanishes at infinity,
then the spectrum of the Dirichlet Laplacian in~$\Omega$, $-\Delta_D^\Omega$,
equals $[E_1,\infty)$.
Consequently, \eqref{Hardy.tube}~implies that this spectrum is stable
against small short-range perturbations of the Laplacian
whenever the tube is locally twisted so that $c_H>0$.
This is the spectral stability, which has applications
to quantum transport in waveguide-shaped nanostructures.

Various generalisations of the Hardy inequality~\eqref{Hardy.tube}
has been established in \cite{K6-with-erratum,KZ1,KR}.
In addition to the quantum-waveguide context,
inequality~\eqref{Hardy.tube} has been also applied to the study of
the large-time behaviour of the heat equation in twisted tubes
in \cite{KZ1,Grillo-Kovarik-Pinchover_2014}.
Other effects of twisting has been studied in
\cite{EKov_2005,K3,Kovarik-Sacchetti_2007,EKov_2005,
Briet-Kovarik-Raikov-Soccorsi_2009,KZ2,KKolb,Briet-Kovarik-Raikov_2014}.

In this paper, we are interested in the existence
of Hardy inequalities in situations
when the tube~\eqref{tube} exhibits a twist
which is not necessarily local,
\ie~$\dot\theta$ may not vanish at infinity.
Indeed, throughout this paper, we assume that
\begin{equation}\label{angle}
  \dot\theta(x_1) = \beta + \eps(x_1)
  \,,
\end{equation}
where~$\beta$ is a real constant and $\eps:\Real\to\Real$ is a
(not necessarily small) bounded function
(typically vanishing at infinity).

If $\eps=0$, then~$\Omega$ is periodically twisted
and the spectral problem can be solved by a Floquet-type decomposition.
It is shown in \cite{EKov_2005,Briet-Kovarik-Raikov-Soccorsi_2009}
that, in this case, $\sigma(-\Delta_D^{\Omega}) = [\lambda_1,\infty)$,
where $\lambda_1$ is the lowest eigenvalue of $-\Delta_D^\omega -
\beta^2 \partial_\tau^2$ in $\sii(\omega)$,
with $\partial_\tau:=x_3\partial_2-x_2\partial_3$
being the transverse angular derivative.
We have the variational characterisation
\begin{equation}\label{lambda}
  \lambda_1 = \inf_{\chi \in C_0^\infty(\omega)\setminus\{0\}}
  \frac{\|\nabla'\chi\|_{\sii(\omega)}^2
  + \beta^2 \;\! \|\partial_\tau\chi \|_{\sii(\omega)}^2}
  {\|\chi\|_{\sii(\omega)}^2}
  \,,
\end{equation}
where $\nabla':=(\partial_2,\partial_3)$
stands for the transverse gradient.
Here and in the sequel,
we keep the coordinate notation introduced in~\eqref{tube},
writing $x':=(x_2,x_3) \in \omega$
for the ``transverse'' coordinates,
while $x_1 \in \Real$ stands for the ``longitudinal'' coordinate.

For $\eps\not=0$ but vanishing at infinity, we always have
(\cf~\cite[Sec.~4.1]{Briet-Kovarik-Raikov-Soccorsi_2009})
\begin{equation}\label{essential}
  \sigma_\mathrm{ess}(-\Delta_D^{\Omega}) = [\lambda_1,\infty)
  \,,
\end{equation}
however, there might be also discrete eigenvalues below~$\lambda_1$.
Indeed, it is shown in \cite{EKov_2005} that the discrete spectrum
is not empty provided that the twist is locally ``slowed down'', \ie,
\begin{equation}\label{discrete}
  \int_\Real \big(\dot\theta^2(x_1) - \beta^2\big) \, dx_1 < 0
  \,.
\end{equation}
Recalling~\eqref{angle},
this condition is for instance true if
$\beta\eps$ is non-positive and not identically equal to zero
and~$\eps$ is small in the supremum norm
with respect to~$\beta$.
Further properties of the discrete spectrum are
studied in \cite{Briet-Kovarik-Raikov-Soccorsi_2009}.

Our objective is to show that there are Hardy-type inequalities
\begin{equation}\label{Hardy.per}
  -\Delta_D^{\Omega} - \lambda_1 \geq \rho(\cdot)
  \,,
\end{equation}
with a non-trivial function $\rho:\Real\to[0,\infty)$
in opposite regimes to~\eqref{discrete}.
In particular, there is no discrete spectrum.
Note that~\eqref{Hardy.tube} is a version of~\eqref{Hardy.per}
if $\beta=0$, since $\lambda_1=E_1$ in this case.
More precisely, we make the following conjectures.

\begin{Conjecture}\label{Conj.sign}
\eqref{Hardy.per} holds if $\beta\eps$ is non-negative
and~$\eps$ is not identically equal to zero.
\end{Conjecture}
\begin{Conjecture}\label{Conj.large}
\eqref{Hardy.per} holds if we replace $\eps \mapsto \alpha \eps$,
$\eps$~is not identically equal to zero
and the coupling parameter~$\alpha$
is sufficiently large in absolute value.
\end{Conjecture}

We say that the twist is \emph{repulsive}
if~$\beta$ and~$\eps$ are such as supposed in Conjecture~\ref{Conj.sign}.
Note that we impose no sign restrictions
in Conjecture~\ref{Conj.large}.

Unfortunately, we have not been able to prove the conjectures
in the full generality.
In this paper, we establish Conjecture~\ref{Conj.sign}
under the additional assumption that the twist is small
in a suitable sense. Among the variety of Hardy inequalities
proved below, we point out the following result here.

\begin{Theorem}\label{Thm.intro}
Let~$\theta$ be given by~\eqref{angle},
where~$\eps$ is a bounded function. Assume $\beta\eps \geq 0$ and
$\beta\eps\not=0$. Suppose that~$\omega$ is not rotationally
invariant and that its boundary~$\partial\omega$ is of class~$C^4$.
There exist positive constants $\beta^*=\beta^*(\omega)$,
$\eps^*=\eps^*(\beta,\eps,\omega)$ and $c=c(\beta,\eps,\omega)$ such
that if $|\beta| \leq \beta^*$ and $\|\eps\|_\infty \leq \eps^*$
then
\begin{equation}\label{Hardy.intro}
  \forall \psi\in H_0^1(\Omega)
  \,, \qquad
  \int_\Omega |\nabla\psi(x)|^2 \, dx
  - \lambda_1 \int_\Omega |\psi(x)|^2 \, dx
  \geq c \int_\Omega \frac{|\psi(x)|^2}{1+|x|^2} \, dx
  \,.
\end{equation}
\end{Theorem}

The validity of Conjecture~\ref{Conj.large} is only supported
by heuristic arguments and numerical experiments
presented in the following section.

The organisation of the paper is as follows.
Section~\ref{Sec.heuristic} is devoted to
mostly non-rigorous arguments supporting the validity
of Conjectures~\ref{Conj.sign} and~\ref{Conj.large}.
Various Hardy inequalities related to Conjecture~\ref{Conj.sign},
in particular that of Theorem~\ref{Thm.intro},
are derived in a long Section~\ref{Sec.proof}
divided into many subsections.
In Appendix~\ref{App} we give a proof of positivity
for a one-dimensional Schr\"odinger operator
(\cf~Proposition~\ref{Prop.1D})
which we use as a support for the validity
of Conjecture~\ref{Conj.large} in Section~\ref{Sec.heuristic}.
Finally, in Appendix~\ref{App.Neumann} we explain
why the tool of Neumann bracketing is not suitable
for the proof of Hardy inequalities in the present setting.

\section{Heuristic arguments and numerics}\label{Sec.heuristic}
%
The validity of Conjectures~\ref{Conj.sign} and~\ref{Conj.large}
is supported by the following arguments.

\subsection{Thin-width asymptotics}
Given a positive number~$\delta$, let us denote by~$\Omega_\delta$
the tube~\eqref{tube} where~$\omega$ is replaced by the scaled
domain $\delta\omega:=\{\delta \:\! x' \,|\, x'\in\omega\}$. The
behaviour of the spectrum of $-\Delta_D^{\Omega_\delta}$ as $\delta
\to 0$ has been investigated in \cite{BMT,deOliveira_2006,KSed}. In
the last reference it is proved that the limit
\begin{equation}\label{effective}
  -\Delta_D^{\Omega_\delta} - \delta^{-2} E_1
  \ \xrightarrow[\delta  \to 0]{} \
  -\Delta^\Real + C_\omega \, \dot\theta^2
\end{equation}
holds in a norm resolvent sense
after a suitable identification of Hilbert spaces.
Here $-\Delta^\Real$ denotes
the one-dimensional Laplacian in $\sii(\Real)$
with $\Dom(-\Delta^\Real):=H^2(\Real)$
and $C_\omega:=\|\partial_\tau\mathcal{J}_1\|_{\sii(\omega)}^2$,
where~$\mathcal{J}_1$ is a normalised eigenfunction of
$-\Delta_D^{\omega}$ corresponding to~$E_1$.
Note that~$C_\omega$ is positive if, and only if,
$\omega$~is not rotationally invariant
with respect to the origin in~$\Real^2$.

Let~$\lambda_1(\delta)$ denote the eigenvalue~\eqref{lambda}
where~$\omega$ is replaced by~$\delta\omega$.
Using the asymptotics
$
  \lambda_1(\delta) = \delta^{-2} E_1 + C_\omega \beta^2
  + \mathcal{O}(\delta^2)
$
as $\delta \to 0$, \eqref{effective} yields
\begin{equation}\label{effective.per}
  -\Delta_D^{\Omega_\delta} - \lambda_1(\delta)
  \ \xrightarrow[\delta  \to 0]{} \
  -\Delta^\Real + C_\omega \, (\dot\theta^2-\beta^2)
\end{equation}
in the norm resolvent sense. It follows that~\eqref{discrete} is
indeed sufficient for the existence of discrete eigenvalues, at
least in the regime of small~$\delta$.
On the other hand, no discrete spectrum is expected if the expression
$
  \dot\theta^2-\beta^2 = \eps^2 + 2\beta\eps
$
is non-negative. This is obviously the case if $\beta\eps$
is non-negative (Conjecture~\ref{Conj.sign}).

At the same time, replacing $\eps \mapsto \alpha\eps$ in~\eqref{angle}
and considering the resulting~$\dot\theta$ in the potential
on the right hand side~\eqref{effective.per},
we end up with the one-dimensional Schr\"odinger operator
\begin{equation}\label{H.1D}
  H_\alpha := -\Delta^\Real + C_\omega \, (\alpha^2\eps^2 + 2\alpha\beta\eps)
\end{equation}
depending on the coupling constant $\alpha \in \Real$.
Under some hypotheses about~$\eps$,
$H_\alpha$~has no negative spectrum if the coupling~$|\alpha|$
is large enough (Conjecture~\ref{Conj.large}).
This is obvious for~$\eps$ of the shape of a characteristic function.
In general, the problem is to take into account the small intervals
around points where (continuous)~$\eps$ vanishes.
In Appendix~\ref{App} we prove the following sufficient condition.
\begin{Proposition}\label{Prop.1D}
Let~$\eps$ be a continuous function whose support is a closure of a
\emph{finite} union of bounded open intervals.
Then there exists a non-negative number~$\alpha_0$
depending on~$\beta$, $C_\omega$ and properties of~$\eps$
such that $H_\alpha \geq 0$ for all $|\alpha| \geq \alpha_0$.
\end{Proposition}

The question is to extend these asymptotic results
to non-infinitesimally small~$\delta$
and to cast the operator positivity
to the existence of Hardy inequalities.

\subsection{Geometric considerations}
The following heuristic argument relies on one's imagination only.
Assume that the function~$\eps$ has a piecewise constant profile:
$\eps(x_1)=\alpha>0$ if $|x_1| < x_1^0$,
with some positive~$x_1^0$, and $\eps(x_1)=0$ elsewhere.

If the origin of $\Real^2$ lies outside~$\omega$,
$\Omega$~converges in a geometric sense as $|\alpha|  \to \infty$
to the disjoint union of two semi-tubes
$
  \Omega_\pm := \Omega \cap \{\pm x_1 > x_1^0\}
$.
By adapting the proof of Section~\ref{Sec.trivial}, it is easy
to see that the spectrum of the Dirichlet Laplacian in any of the
semi-tubes does not start below~$\lambda_1$. Moreover, because of
the presence of the extra Dirichlet conditions at $\{|x_1|=x_1^0\}$
and~\eqref{Hardy.1D},
the shifted operators $-\Delta_D^{\Omega_\pm}-\lambda_1$
will satisfy a Hardy inequality of the type~\eqref{Hardy.per}.

If the origin of~$\Real^2$ lies inside~$\omega$,
$\Omega$~converges as $\alpha \to \infty$
to a set composed of~$\Omega_-$, $\Omega_+$
and a connecting tubular channel of radius $\dist(0,\partial\omega)$.
Again, it should be possible to show that the Dirichlet Laplacian
shifted by~$\lambda_1$ satisfies a Hardy inequality in this domain.

The above arguments give a strong geometric support
for the validity of Conjecture~\ref{Conj.large},
at least in the case of the special profile of~$\eps$.

\subsection{Numerical simulations}
Finally, we have performed several numerical experiments
to support the validity of
Conjectures~\ref{Conj.sign} and~\ref{Conj.large}.
For illustration, let us take a square cross-section
$\omega = (-\frac{1}{2},\frac{1}{2})^2$ and $\beta=1$.
In front of~$\eps$ in~\eqref{angle},
we add the coupling constant $\alpha \in \Real$
and consider two kinds of profile of~$\eps$:
\begin{equation}\label{kinds}
  \eps_1(x_1) :=
  \begin{cases}
    1 & \mbox{if} \quad |x_1| \leq 1 \,,
    \\
    0 & \mbox{elsewhere} \,,
  \end{cases}
  \qquad
  \eps_2(x_1) :=
  \begin{cases}
    1-|x_1| & \mbox{if} \quad |x_1| \leq 1 \,,
    \\
    0 & \mbox{elsewhere} \,.
  \end{cases}
\end{equation}
Our numerical calculations have been done in a finite tube
$
  \Omega_L := \Omega \cap \{|x_1| < L\}
$
with~$L$ so large that the computed eigenvalues stop
to be sensitive to further enlargements of~$L$.

In Figure~\ref{fig} we present the dependence of
the lowest eigenvalue of~$-\Delta_D^{\Omega_L}$ on~$\alpha$
as the blue curve.
The horizontal red line corresponds to the energy~$\lambda_1$,
which is the threshold of the essential spectrum of~$-\Delta_D^{\Omega}$,
\cf~\eqref{essential}.
Hence, the blue curve below the red horizontal line
approximates the lowest discrete eigenvalue of~$-\Delta_D^{\Omega}$,
while there is just the essential spectrum of~$-\Delta_D^{\Omega}$
above the red line.
Consequently, the validity of a Hardy-type inequality~\eqref{Hardy.per}
is expected whenever the blue curve is strictly above the red line.

We see that the numerical pictures clearly confirm
our Conjectures~\ref{Conj.sign} and~\ref{Conj.large}.
Indeed, $\inf\sigma(-\Delta_D^{\Omega_L})>\lambda_1$
whenever $\alpha>0$ (Conjecture~\ref{Conj.sign})
or $\alpha < \alpha^* < 0$ (Conjecture~\ref{Conj.large}).
We approximately have $\alpha_1^*\approx-2$ and $\alpha_2^*\approx-3$
for the profile~$\eps_1$ and~$\eps_2$, respectively.
It is remarkable that the critical values of~$\alpha^*$
are so close to the smallest values of~$\alpha$
for which the sufficient condition~\eqref{discrete} applies.
Indeed, \eqref{discrete}~yields the existence of
a discrete spectrum of~$-\Delta_D^\Omega$
whenever $\alpha \in (-2,0)$ and $\alpha \in (-3,0)$
for the profile~$\eps_1$ and~$\eps_2$, respectively.

\begin{figure}[h!]
\begin{center}
\includegraphics[width=1.0\textwidth]{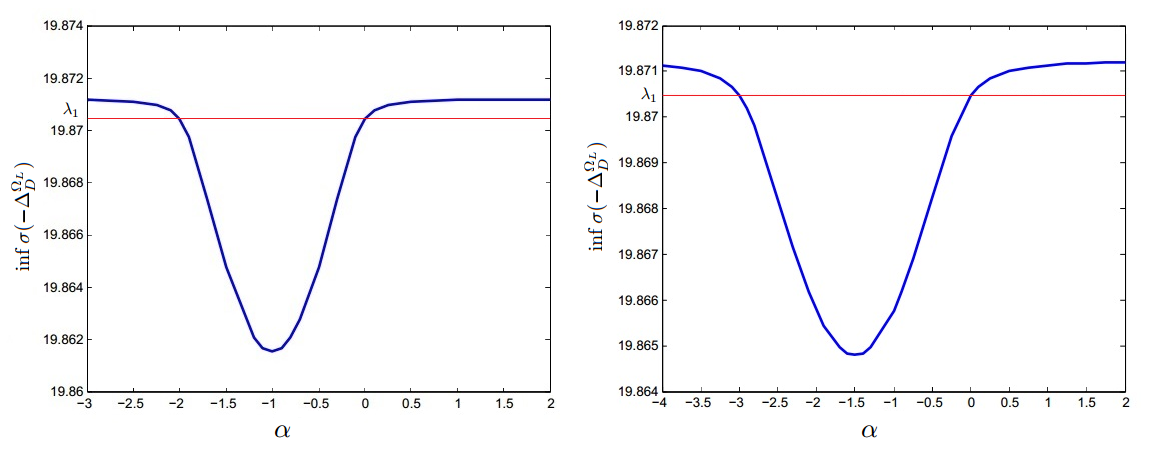}
\end{center}
\caption{Dependence of
$\inf\sigma(-\Delta_D^{\Omega_L})$ on~$\alpha$
for the two profiles~\eqref{kinds} and $L=100$.
The horizontal (red) line corresponds to the energy
$\lambda_1=\inf\sigma_\mathrm{ess}(-\Delta_D^\Omega)$.}\label{fig}
\end{figure}
%

\section{Hardy inequalities for a repulsive twist}\label{Sec.proof}
%
This long section divided into many subsections
is primarily intended to establish
a well-arranged proof of Theorem~\ref{Thm.intro},
which deals with $\beta\eps \geq 0$.
However, some of the intermediate results
might be interesting on its own
and without this sign (and other) restriction(s).

\subsection{Curvilinear coordinates}
The very definition~\eqref{tube} gives rise to
a diffeomorphism between~$\Omega$
and the straight tube~$\Real\times\omega$.
Passing to the curvilinear coordinates
$(x_1,x_2,x_3) \in \Real \times \omega$,
the Dirichlet Laplacian $-\Delta_D^\Omega$ in $\sii(\Omega)$
can be identified
(\cf~\cite{K6-with-erratum} for more details)
with the operator~$H$ in $\sii(\Real\times\omega)$
associated with the quadratic form
\begin{equation}\label{form}
  h[\psi] := \|\partial_1\psi-\dot\theta\partial_\tau\psi\|^2
  + \|\nabla'\psi\|^2
  \,, \qquad
  \Dom(h) := H_0^1(\Real\times\omega)
  \,.
\end{equation}
Here and in the sequel $\|\cdot\|$
denotes the norm of $\sii(\Real\times\omega)$.
The associated inner product will be denoted by $(\cdot,\cdot)$.

Since~$\dot\theta$ is bounded,
the space $C_0^\infty(\Real\times\omega)$ is a core of~$h$.
Henceforth we thus take an arbitrary
$$
  \psi \in C_0^\infty(\Real\times\omega)
  \,.
$$
Moreover, since~$H$ commutes with complex conjugation,
we may suppose that~$\psi$ is real-valued.

\subsection{Ground-state decomposition}\label{Sec.GSD}
Let~$\chi$ denote an eigenfunction of
$-\Delta_D^\omega - \beta^2 \partial_\tau^2$
corresponding to~$\lambda_1$.
We choose~$\chi$ positive and normalised to~$1$ in $\sii(\omega)$.
Since we are interested in properties of~$H$ near the
threshold~$\lambda_1$ of the essential spectrum,
it is useful to make the decomposition
\begin{equation}\label{decomposition}
  \psi(x) = \chi(x') \phi(x)
  \,,
\end{equation}
where~$\phi$ is a $C_0^\infty(\Real\times\omega)$
function actually defined by~\eqref{decomposition}.

It is straightforward to check that
\begin{equation}\label{form.shifted}
\begin{aligned}
  Q[\psi] :=& \ h[\psi] - \lambda_1 \|\psi\|^2
  \\
  =& \
  \|\partial_1\psi-\eps\partial_\tau\psi-\beta\chi\partial_\tau\phi\|^2
  + \|\chi\nabla'\phi\|^2
  - 2 (\partial_1\psi-\eps\partial_\tau\psi,\beta\phi\partial_\tau\chi)
  \\
  =& \
  \|\partial_1\psi\|^2 + \|\eps\partial_\tau\psi\|^2
  + \|\chi\nabla'\phi\|^2 + \|\beta\chi\partial_\tau\phi\|^2
  \\
  & \
  + 2 (\partial_\tau\psi,\beta\eps\partial_\tau\psi)
  - 2 (\partial_1\psi,\eps\partial_\tau\psi)
  - 2 (\partial_1\psi,\beta\partial_\tau\psi)
  \,.
\end{aligned}
\end{equation}
Here and in the sequel, we use the same symbol~$\eps$
(respectively~$\chi$) for the function $\eps\otimes 1$ (respectively
$1\otimes\chi$) on $\Real\times\omega$, and similarly for other
functions that will appear below.

\subsection{Positivity for a trivial twist}\label{Sec.trivial}
It is not clear from~\eqref{form.shifted}
whether $Q[\psi] \geq 0$ if $\beta\eps \geq 0$.
In fact, the non-negativity is not completely obvious
even for the trivial situation $\eps=0$
(periodically twisted tube),
but it can be established as follows.

If $\eps=0$, then~\eqref{form.shifted} reduces to
\begin{equation}\label{form.zero}
  Q[\psi] =
  \|\partial_1\psi\|^2
  + \|\chi\nabla'\phi\|^2 + \|\beta\chi\partial_\tau\phi\|^2
  - 2 (\partial_1\psi,\beta\partial_\tau\psi)
  \,.
\end{equation}
We write the mixed term in~\eqref{form.zero} as follows
\begin{equation}\label{zero1}
  - 2 (\partial_1\psi,\beta\partial_\tau\psi)
  = - 2 (\partial_1\psi,\beta\chi\partial_\tau\phi)
  - 2 (\partial_1\psi,\beta\phi\partial_\tau\chi)
  \,.
\end{equation}
By an integration by parts in the first variable, the second term on
the right hand side is equal to zero:
\begin{equation}\label{zero2}
  - 2 (\partial_1\psi,\beta\phi\partial_\tau\chi)
  =
  - \frac{1}{2} \, \beta \int_{\Real\times\omega}
  (\partial_1\phi^2) (\partial_\tau\chi^2)
  = 0
  \,.
\end{equation}
Consequently,
\begin{equation}\label{xx}
  Q[\psi] =
  \|\partial_1\psi-\beta\chi\partial_\tau\phi\|^2
  + \|\chi\nabla'\phi\|^2
  \geq 0
  \,.
\end{equation}
That is, $H-\lambda_1 \geq 0$.

Except for $\psi=0$ (\ie~$\phi=0$),
inequality~\eqref{xx} is always strict.
On the other hand, \eqref{xx}~becomes sharp asymptotically
as $N\to\infty$ when considering the sequence of functions
$\psi_N(x)=\phi_N(x_1)\chi(x')$,
where $\phi_N(x_1):=1$ if $|x_1| \leq N$,
$\phi_N(x_1):=(2N-|x_1|)/N$ if $N < |x_1| < 2N$,
and $\phi_N=0$ otherwise.
Consequently, $H-\lambda_1$ is critical in the sense that
adding to $H-\lambda_1$
an arbitrarily small non-positive smooth potential
which is not identically equal to zero
leads to the appearance of negative spectrum.
This shows that we cannot have a Hardy inequality~\eqref{Hardy.per}
for $\eps=0$.

\subsection{Positivity for a small repulsive twist}
Using~\eqref{zero2}, we rewrite the second line
of~\eqref{form.shifted} as follows
\begin{equation}\label{but.one}
   Q[\psi] =
   \|\partial_1\psi-\eps\partial_\tau\psi-\beta\chi\partial_\tau\phi\|^2
    + \|\chi\nabla'\phi\|^2
    + 2 (\eps\phi\partial_\tau\chi,\beta\phi\partial_\tau\chi)
    + 2 (\eps\chi\partial_\tau\phi,\beta\phi\partial_\tau\chi)
  \,.
\end{equation}
Note that the last but one integral on the right hand side is
non-negative whenever $\beta\eps \geq 0$,
in particular for any repulsive twist.
We assume this sign restriction henceforth.
The last integral (of indefinite sign)
can be estimated by means of the Schwarz and Young inequalities
\begin{equation}\label{Young}
\begin{aligned}
  2 (\eps\chi\partial_\tau\phi,\beta\phi\partial_\tau\chi)
  &\geq -2 \big\|\sqrt{\beta\eps}\,\phi\partial_\tau\chi\big\|
  \big\|\sqrt{\beta\eps}\,\chi\partial_\tau\phi\big\|
  \\
  &\geq - \delta \big\|\sqrt{\beta\eps}\,\phi\partial_\tau\chi\big\|^2
  - \frac{1}{\delta} \big\|\sqrt{\beta\eps}\,\chi\partial_\tau\phi\big\|^2
  \,,
\end{aligned}
\end{equation}
with any positive~$\delta$. Here the first term on the right hand
side can be controlled by the last but one integral on the right
hand side of~\eqref{but.one}. The second term on the right hand side
of~\eqref{Young} can be estimated using the pointwise estimate
\begin{equation}\label{pointwise}
  |\partial_\tau\phi| \leq a \, |\nabla'\phi|
  \,, \qquad
  a := \sup_{x'\in\omega} |x'|
  \,,
\end{equation}
and controlled by the second term on the right hand side
of~\eqref{but.one} provided that~$\beta\eps$ is small.
More specifically, we thus have
\begin{equation}\label{crucial}
  Q[\psi] \geq
  \|\partial_1\psi-\eps\partial_\tau\psi-\beta\chi\partial_\tau\phi\|^2
  + (2-\delta) \big\|\sqrt{\beta\eps}\,\phi\partial_\tau\chi\big\|^2
  + \left[1 - \frac{\|\beta\eps\|_\infty \, a^2}{\delta} \right]
  \|\chi\nabla'\phi\|^2
  \,.
\end{equation}
Consequently, choosing~$\delta=2$,
we conclude with the desired positivity:
\begin{Proposition}
$H-\lambda_1 \geq 0$ provided that $\beta\eps \geq 0$ and
\begin{equation}\label{small.hypotheses}
  \|\beta\eps\|_\infty \, a^2 \leq 2
  \,.
\end{equation}
\end{Proposition}

\subsection{Local Hardy inequality}
In addition to $\beta\eps \geq 0$,
let us now assume that~$\beta\eps$ is non-trivial,
so that we are in the situation of repulsive global twist
(Hardy inequalities for the local twist, \ie\ $\beta=0$,
such as~\eqref{Hardy.tube} are known, see the introduction).
We also strengthen~\eqref{small.hypotheses} to
\begin{equation}\label{small.hypotheses.strict}
  \|\beta\eps\|_\infty \, a^2 < 2
  \,.
\end{equation}
Choosing in~\eqref{crucial} $\delta =  \|\beta\eps\|_\infty \, a^2$,
neglecting the first term on the right hand side and recalling the
decomposition~\eqref{decomposition}, we get
\begin{equation}\label{pre-Hardy}
  Q[\psi] \geq
  (2-\|\beta\eps\|_\infty \, a^2)
  \left\|\sqrt{\beta\eps}\,\frac{\partial_\tau\chi}{\chi} \psi\right\|^2
  \,.
\end{equation}
This inequality has been established
for any $\psi \in C_0^\infty(\Real\times\omega)$,
however, by density it extends to all $\psi
\in H_0^1(\Real\times\omega)$.
It thus represents a Hardy-type inequality~\eqref{Hardy.per}
that we state in the following theorem.
\begin{Theorem}\label{Thm.local}
Let $\beta\eps \geq 0$ and $\beta\eps\not=0$.
Then
\begin{equation}\label{HI.local}
  H-\lambda_1 \geq
  (2-\|\beta\eps\|_\infty \, a^2)
  \, \beta \eps
  \left(\frac{\partial_\tau\chi}{\chi}\right)^2
  \,.
\end{equation}
\end{Theorem}

It is a \emph{local} Hardy inequality
(\cf~\cite{K6-with-erratum} for the terminology)
if~\eqref{small.hypotheses.strict} holds,
since~$\eps$ can be compactly supported in~$\Real$.
Moreover, $\partial_\tau\chi$ can vanish in~$\overline{\omega}$.
However, it is important to notice that~$\partial_\tau\chi$ cannot
vanish on a subset of~$\omega$ with positive measure,
unless~$\omega$ is rotationally invariant. Indeed, by
differentiating the equation for~$\chi$ and noticing
that~$\partial_\tau$ commutes with~$\Delta'$, the function
$\eta:=\partial_\tau\chi$ satisfies the same equation
$
  (-\Delta' - \beta^2 \partial_\tau^2) \eta
  = \lambda_1 \eta
$
in~$\omega$ for which the unique continuation property holds.
Finally, let us notice that the function
on the right hand side of~\eqref{HI.local}
can diverge on~$\partial\Omega$
due to the presence of~$\chi$ in the denominator,
which makes a resemblance to the classical
Hardy inequality~\eqref{Hardy.1D}.

\subsection{A very brute estimate}
We continue assuming $\beta\eps \geq 0$ and $\beta\eps\not=0$.
Our objective is to
cast~\eqref{HI.local} into a \emph{global} Hardy inequality,
\ie~with a right hand side being a positive function
in $\Real\times\omega$.
This can be done by employing the presence of~$\|\partial_1\psi\|$
and~$\|\chi\nabla'\phi\|$ in~\eqref{crucial}.

We thus come back to the second equality in~\eqref{form.shifted},
employ~\eqref{zero2} and further develop the expression as follows
\begin{equation}\label{form.shifted.bis}
\begin{aligned}
  Q[\psi] \ =& \
  \ \|\partial_1\psi\|^2
  + \|\eps\chi\partial_\tau\phi\|^2
  + \|\eps\partial_\tau\chi \phi\|^2
  + \|\chi\nabla'\phi\|^2 + \|\beta\chi\partial_\tau\phi\|^2
  \\
  & \
  + 2 \big\|\sqrt{\beta\eps} \chi \partial_\tau\phi\big\|^2
  +  2  \big\|\sqrt{\beta\eps}\partial_\tau\chi \phi\big\|^2
  +  4  \big(\sqrt{\beta\eps}\chi \partial_\tau\phi,
  \sqrt{\beta\eps}\partial_\tau\chi \phi\big)
  \\
  & \
  + 2 (\eps\chi\partial_\tau\phi,\eps\partial_\tau\chi \phi)
  - 2 (\partial_1\psi, \eps\partial_\tau\chi \phi)
  - 2 \big(\partial_1\psi,(\eps+\beta)\chi\partial_\tau\phi\big)
  \,.
\end{aligned}
\end{equation}
Using the brute estimates
\begin{align*}
  | 2 (\partial_1\psi,\eps\partial_\tau\chi \phi) |
  & \leq \delta_1 \|\partial_1\psi\|^2
  + \delta_1^{-1} \|\eps\partial_\tau\chi \phi\|^2
  \,,
  \\
  | 2 \big(\partial_1\psi,(\eps+\beta)\chi\partial_\tau\phi\big) |
  & \leq \delta_2 \|\partial_1\psi\|^2
  + \delta_2^{-1} \|(\eps+\beta)\chi\partial_\tau\phi\|^2
  \,,
  \\
  \big| 2 \big(\sqrt{\beta\eps}\chi \partial_\tau\phi,
  \sqrt{\beta\eps}\partial_\tau\chi \phi\big) \big|
  & \leq \delta_3 \big\|\sqrt{\beta\eps}\partial_\tau\chi \phi\big\|^2
  + \delta_3^{-1} \big\|\sqrt{\beta\eps}\chi \partial_\tau\phi\big\|^2
  \,,
  \\
  | 2 (\eps\chi\partial_\tau\phi,\eps\partial_\tau\chi \phi) |
  & \leq \delta_4 \|\eps\partial_\tau\chi \phi\|
  + \delta_4^{-1} \big\|\eps\chi\partial_\tau\phi\|^2
  \,,
\end{align*}
with arbitrary positive numbers $\delta_1,\delta_2,\delta_3,\delta_4$,
we obtain
\begin{align*}
Q[\psi] \geq \ &
(1-\delta_{1}-\delta_{2})\|\partial_{1}\psi\|^{2}
+\|\chi\nabla'\phi\|^{2}
\\
&+ \int_{\Real\times\omega}
\left[
(1-\delta_{1}^{-1}-\delta_{4})\eps^{2}
+2(1-\delta_{3})\beta\eps
\right]
|\phi\partial_{\tau}\chi|^{2} \, dx
\\
&+ \int_{\Real\times\omega}
\left[
(1-\delta_{2}^{-1})\beta^{2}
+(1-\delta_{2}^{-1}-\delta_{4}^{-1})\eps^{2}
+2(1-\delta_{2}^{-1}-\delta_{3}^{-1})\beta\eps
\right]
|\chi\partial_{\tau}\phi|^{2} \, dx.
\end{align*}
Choosing $\delta_1=\delta_2=1/4$, $\delta_3=1/2$ and $\delta_4=1$,
the previous inequality reads
\begin{align*}
  Q[\psi] \geq \ & \frac{1}{2}\|\partial_{1}\psi\|^{2}
  + \|\chi\nabla'\phi\|^2
  + \int_{\Real\times\omega}(\beta\eps- 4\eps^{2})
  |\phi\partial_{\tau}\chi|^{2} \, dx
  \\
  & - \int_{\Real\times\omega}
  (3\beta^{2}+4\eps^{2}+10\beta\eps)
  |\chi\partial_\tau\phi|^2 \, dx
  \,.
\end{align*}
Finally, employing the pointwise bound~\eqref{pointwise},
we conclude with
\begin{equation}\label{Q12}
  Q[\psi] \geq
  \ \frac{1}{2} \|\partial_1\psi\|^2
  + c_1 \|\chi\nabla'\phi\|^2
  + c_2 \big\|\sqrt{\beta\eps}\;\!\partial_\tau\chi \phi\big\|^2
  \,,
\end{equation}
where
\begin{equation}\label{constants}
  c_{1}=1-a^{2}\,
  (3\beta^{2}+4\|\eps\|_{\infty}^{2}+10\|\beta\eps\|_{\infty})
  \,,
  \qquad
  c_{2}=1-4\frac{\|\eps\|_{\infty}}{|\beta|}
  \,,
\end{equation}
are positive constants provided that 
\begin{equation}\label{Ass}
  4 \, \|\eps\|_\infty < |\beta| < 
  \frac{2}{a \sqrt{23}} 
  \,.
\end{equation}
This is a condition on the smallness of both
the global periodic twist~$\beta$
and its local perturbation~$\eps$.

\subsection{An auxiliary transverse problem}
For any number $\epsilon \in \Real$, define
\begin{equation}\label{mu}
  \mu_\epsilon := \inf_{\phi \in C_0^\infty(\omega)\setminus\{0\}}
  \frac{\|\chi\nabla'\phi\|_{\sii(\omega)}^2
  + \epsilon^2 \|\partial_\tau\chi \phi\|_{\sii(\omega)}^2}
  {\|\chi \phi\|_{\sii(\omega)}^2}
  \,.
\end{equation}
Consider the quadratic form
$$
  q_\epsilon[\phi] :=
  \|\chi\nabla\phi\|_{\sii(\omega)}^2
  + \epsilon^2 \|\partial_\tau\chi \phi\|_{\sii(\omega)}^2
  \,, \qquad
  \Dom(q_\epsilon) := C_0^\infty(\omega)
  \,,
$$
in the Hilbert space $\sii(\omega,\chi(x')^2 dx')$ and denote by
$\tilde{q}_\epsilon$ its closure. Then $\mu_\epsilon$ is the lowest
point in the spectrum of the self-adjoint operator~$h_\epsilon$ in
$\sii(\omega,\chi(x')^2 dx')$ associated with~$\tilde{q}_\epsilon$.
Our objective is to show that $\mu_\epsilon$ is positive unless
$\epsilon=0$ or~$\omega$ is rotationally symmetric with respect to
the origin.

\begin{Lemma}\label{Lem.compact}
Let~$\partial\omega$ be of class~$C^4$. There exists a
positive constant~$\epsilon_0$, depending on the geometry of~$\omega$, such
that for all $|\epsilon| < \epsilon_0$, $h_\epsilon$ is an operator
with compact resolvent.
\end{Lemma}
\begin{proof}
Let us introduce the unitary transform
$$
  U: \sii(\omega,\chi(x')^2 dx') \to \sii(\omega):
  \{\phi \mapsto \chi\phi\}
  \,,
$$
which is well defined because~$\chi$ is positive in~$\omega$.
Then~$h_\epsilon$ is unitarily equivalent to the operator
$\hat{h}_\epsilon := U h_\epsilon U^{-1}$ in $\sii(\omega)$.
The latter is the operator associated in $\sii(\omega)$ with the
quadratic form
$$
  \hat{q}_\epsilon[\psi] := \tilde{q}_\epsilon[U^{-1}\psi]
  \,, \qquad
  \Dom(\hat{q}_\epsilon) := U\Dom(\tilde{q}_\epsilon)
  \,.
$$
Notice that the space $C_0^\infty(\omega)$, which is a core of
$\tilde{q}_\epsilon$, is left invariant by both~$U$ and~$U^{-1}$.
For any $\psi \in C_0^\infty(\omega)$, by integrating by parts, it
is easy to verify
$$
  \hat{q}_\epsilon[\psi] =
  \|\nabla\psi\|_{\sii(\omega)}^2
  + (\psi,V\psi)_{\sii(\omega)}
  \,,
$$
where
$$
  V := \frac{\Delta\chi}{\chi}
  + \epsilon^2 \left(\frac{\partial_\tau\chi}{\chi}\right)^2
  \,.
$$
Without the potential~$V$, $\hat{q}_\epsilon$ would be just the form
associated with the Dirichlet Laplacian $-\Delta_D^\omega$
in~$\sii(\omega)$. The form domain of the latter is $H_0^1(\omega)$,
which is compactly embedded in~$\sii(\omega)$. It is thus enough to
show that~$V$ is a relatively form bounded perturbation of $-\Delta_D^\omega$
due to the stability result \cite[Thm.~VI.3.4]{Kato}.
To do so, we use several facts:
\begin{enumerate}
\item[(i)]
By standard elliptic regularity theory (see, \eg, \cite[Sec.~6.3]{Evans}),
we have $\chi \in H^4(\omega)$.
Consequently, $\nabla\chi \in H^3(\omega)$ and $\Delta\chi \in H^2(\omega)$.
Using the Sobolev embedding \cite[Thm.~5.4]{Adams}
$H^2(\omega) \hookrightarrow C^0(\overline{\omega})$,
we thus have
$\|\Delta\chi\|_\infty < \infty$
and $\|\partial_\tau\chi\|_\infty \leq a \|\nabla\chi\|_\infty < \infty$.
\item[(ii)]
For any domain~$\omega$ such that~$\partial\omega$ is of
class~$C^2$, there exists \cite[Lem.~4.6.1]{Davies_1989}
a positive number~$\alpha_0$ such that $\chi \geq \alpha_0 d$,
where $d(x'):=\dist(x',\partial\omega)$.
\item[(iii)]
For any strongly regular domain~$\omega$, which is in particular
satisfied under the present smoothness assumption, the Hardy
inequality $-\Delta_D^\omega \geq c_0/d^2$ holds true
\cite[Sec.~1.5]{Davies_1989}.
\end{enumerate}

Using~(i)--(ii), we have
\begin{align*}
  \left| \int_\omega \frac{\Delta\chi}{\chi} \, \psi^2 \right|
  &\leq \frac{\|\Delta\chi\|_\infty}{\alpha_0}
  \int_\omega \frac{\psi^2}{d}
  \leq \frac{\|\Delta\chi\|_\infty}{\alpha_0}
  \left(\delta \int_\omega \frac{\psi^2}{d^2}
  + \delta^{-1} \int_\omega \psi^2 \right)
  \,,
  \\
  \left| \int_\omega \epsilon^2
  \left(\frac{\partial_\tau\chi}{\chi}\right)^2 \psi^2 \right|
  &\leq \epsilon^2 \frac{\|\partial_\tau\chi\|_\infty^2}{\alpha_0^2}
  \int_\omega \frac{\psi^2}{d^2}
  \,,
\end{align*}
for any $\psi \in C_0^\infty(\omega)$ and $\delta>0$. 
Finally, using~(iii), we deduce
$$
  \left| \int_\omega V \psi^2 \right|
  \leq
  b \|\nabla\psi\|_{\sii(\omega)}^2
  + C \|\psi\|_{\sii(\omega)}^2
  \,,
$$
where
$$
  b := \frac{1}{c_0} \left(
  \delta \frac{\|\Delta\chi\|_\infty}{\alpha_0}
  + \epsilon^2 \frac{\|\partial_\tau\chi\|_\infty^2}{\alpha_0^2}
  \right)
  \,, \qquad
  C := \delta^{-1} \frac{\|\Delta\chi\|_\infty}{\alpha_0}
  \,.
$$
Hence, by taking~$\delta$ and~$\epsilon$ small enough, $V$~is a
relatively form bounded perturbation of~$-\Delta_D^\omega$ with the
relative bound~$b$ less than one.
\end{proof}
\begin{Proposition}\label{Prop.mu}
Under the hypothesis and notation of Lemma~\ref{Lem.compact},
$\mu_\epsilon > 0$ for every $|\epsilon| \in (0,\epsilon_0)$,
unless~$\omega$ is rotationally symmetric.
\end{Proposition}
\begin{proof}
If $|\epsilon| < \epsilon_0$, the spectrum of~$h_\epsilon$ is
purely discrete. In particular, the spectral
threshold~$\mu_\epsilon$ is an eigenvalue and the infimum~\eqref{mu}
is attained by a corresponding eigenfunction $\phi \in
\sii(\omega,\chi(x')^2 dx')$, \ie,
$$
  \mu_\epsilon =
  \frac{\|\chi\nabla\phi\|_{\sii(\omega)}^2
  + \epsilon^2 \|\partial_\tau\chi \phi\|_{\sii(\omega)}^2}
  {\|\chi \phi\|_{\sii(\omega)}^2}
  \,.
$$
Assuming $\mu_\epsilon = 0$, it follows that
$\|\chi\nabla\phi\|_{\sii(\omega)}=0$ and $\|\partial_\tau\chi
\phi\|_{\sii(\omega)}=0$. From the first identity, since~$\chi$ is
positive, we deduce that~$\phi$ is constant. Putting this result
into the second identity, we conclude with
$\|\partial_\tau\chi\|_{\sii(\omega)}=0$, which is possible only
if~$\omega$ is rotationally symmetric with respect to the origin.
\end{proof}

\subsection{Uniform positivity in the cross-section}
We come back to~\eqref{Q12} and choose the parameters~$\eps$,
$\beta$ and~$a$ in such a way that~$c_1$ and~$c_2$ are positive.
Note that the constants~$c_1$ and~$c_2$ can become only more
positive if~$\|\epsilon\|_\infty$ further diminishes. Employing the
definition~\eqref{mu} and Fubini's theorem, we get
\begin{equation}\label{Q12.mu}
  Q[\psi] \geq
  \frac{1}{2} \|\partial_1\psi\|^2
  + \big\|\sqrt{\mu} \, \psi\big\|^2
  \,,
\end{equation}
where
$$
  \mu(x) := c_1 \, \mu_{\epsilon(x_1)}
  \qquad \mbox{with} \qquad
  \epsilon(x_1) :=  \sqrt{\frac{c_2}{c_1} \, \beta\eps(x_1) }
  \,.
$$
Let~$\omega$ be different from a disc or annulus. Let $|\eps|$ be
non-trivial and so small on a bounded interval $I \subset \Real$
such that $|\epsilon(x_1)| < \epsilon_0$
for almost every $x_1 \in I$. Then
we know by Proposition~\ref{Prop.mu} that $x_1 \mapsto \mu(x_1,x')$
is non-trivial and non-negative on~$I$ (by definition, $\mu(x)$ is
independent of~$x' \in \omega$). Consequently,
\begin{equation}\label{Q12.nu}
  Q[\psi] \geq
  \frac{1}{4} \|\partial_1\psi\|^2
  + \nu \|\psi\|_{\sii(I\times\omega)}^2
  \,,
\end{equation}
where~$\nu$ is the lowest eigenvalue of the one-dimensional
operator $-\frac{1}{4} \Delta_N^I + \mu$ in $\sii(I)$.
Note that~$\nu$ is positive because the potential~$\mu$
is non-trivial and non-negative.
Summing up, we have established the following crucial result.
\begin{Theorem}\label{Thm.local.better}
Let $\beta\eps \geq 0$ and $\beta\eps\not=0$.
Suppose that~$\omega$ is not rotationally invariant
and that its boundary~$\partial\omega$ is of class~$C^4$.
In addition to~\eqref{Ass}, assume that there exists
a bounded interval $I \subset \Real$ such that
\begin{equation}\label{Ass.better}
  0 < |\beta \epsilon(x_1)| < \frac{c_1}{c_2} \, \epsilon_0^2
  \qquad \mbox{for a.e.} \
  x_1 \in I
  \,,
\end{equation}
where~$c_1, c_2$ are the constants depending
on~$a$, $|\beta|$ and $\|\eps\|_\infty$
introduced in~\eqref{constants}
and~$\epsilon_0$ is the number depending
on the geometry of~$\omega$ from Lemma~\ref{Lem.compact}.
Then~\eqref{Q12.nu} holds
for every $\psi \in H_0^1(\Real\times\omega)$
with a positive number~$\nu$.
\end{Theorem}

Notice that~\eqref{Q12.nu} is equivalent to the operator inequality
\begin{equation}\label{HI.local.better}
  H-\lambda_1 \geq
  \left(-\mbox{$\frac{1}{4}$} \Delta^\Real + \nu\,\chi_{I}\right)
  \otimes 1
  \,,
\end{equation}
where~$\chi_I$ is the characteristic function of~$I$
and the right hand side employs the Hilbert-space identification
$\sii(\Real\times\omega) \simeq \sii(\Real) \otimes \sii(\omega)$.
In particular, whenever~$\nu$ is positive,
we get another local Hardy-type inequality
\begin{equation}\label{HI.local.better.bis}
  H-\lambda_1 \geq
  \nu \,\chi_{I\times\omega}
  \,.
\end{equation}

\subsection{Global Hardy inequality}
It is known how to deduce from~\eqref{HI.local.better}
a global Hardy inequality
with help of the classical result~\eqref{Hardy.1D}
(see~\cite{K3} or \cite[Sec.~7.2]{KKolb}).
For the convenience of the reader and self-consistency,
we repeat the procedure here.
\begin{Theorem}\label{Thm.global}
Under the hypotheses of Theorem~\ref{Thm.local.better},
there exists a positive constant~$c$
depending on~$\beta$, the geometry of~$\omega$
and properties of~$\eps$ such that
\begin{equation}\label{Hardy.global}
  H - \lambda_1 \geq \frac{c}{1+x_1^2}
\end{equation}
holds in the form sense in $\sii(\Real\times\omega)$.
\end{Theorem}
\begin{proof}
Let~$x_1^0$ denote the centre of the interval~$I$. 
The main ingredient in the proof is
the following Hardy-type inequality for a Schr\"odinger operator in
$\Real\times\omega$ with a characteristic-function potential:
\begin{equation}\label{Hardy.classical}
  \|\rho\psi\|^2
  \leq 16 \, \|\partial_1\psi\|^2
  + (2+64/|I|^2) \, \|\psi\|_{\sii(I\times\omega)}^2
\end{equation}
for every $\psi \in H_0^1(\Real\times\omega)$,
where we denote $\rho(x):=1/\sqrt{1+(x_1-x_1^0)^2}$.
This inequality is a consequence of~\eqref{Hardy.1D}.
Indeed, following~\cite[Sec.~3.3]{EKK},
let~$\eta$ be the Lipschitz function on~$\Real$ defined by
$\eta(x_1):=2|x_1-x_1^0|/|I|$ for $|x_1-x_1^0|\leq |I|/2$
and~$1$ otherwise.
For any $\psi \in C_0^\infty(\Real\times\omega)$,
let us write $\psi =\eta\psi+(1-\eta)\psi$,
so that $(\eta\psi)(\cdot,x') \in H_0^1(\Real\!\setminus\!\{x_0^1\})$
for every $x'\in\omega$.
Then, employing Fubini's theorem, we can estimate as follows:
\begin{align*}
  \|\rho\psi\|^2
  & \leq 2 \int_{\Real\times\omega} (x_1-x_1^0)^{-2} \, |(\eta\psi)(x)|^2 \, dx
  + 2 \, \|(1-\eta)\psi\|^2
  \\
  & \leq 8 \, \|\partial_1(\eta\psi)\|^2
  + 2 \, \|\psi\|_{\sii(I\times\omega)}^2
  \\
  & \leq 16 \, \|\eta \partial_1\psi\|^2
  + 16 \, \|(\partial_1{\eta})\psi\|^2
  + 2 \, \|\psi\|_{\sii(I\times\omega)}^2
  \\
  & \leq 16 \, \|\partial_1\psi\|^2
  + (2+64/|I|^2) \, \|\psi\|_{\sii(I\times\omega)}^2
  \,.
\end{align*}
By density, this result extends to all
$\psi\in H_0^1(\Real\times\omega)=\Dom(Q)=\Dom(h)$.

By Theorem~\ref{Thm.local.better},
we have
\begin{equation*}
  Q[\psi] \geq \frac{1-\delta}{4} \|\partial_1\psi\|^2
  + \frac{\delta}{4} \|\partial_1\psi\|^2
  + \nu \|\psi\|_{\sii(I\times\omega)}^2
\end{equation*}
for every $\psi \in H_0^1(\Real\times\omega)$ and $\delta\in(0,1]$,
where~$\nu$ is a positive number.
Neglecting the first term on the right hand side
and using~\eqref{Hardy.classical}, we get
\begin{equation*}
  Q[\psi] \geq \frac{\delta}{64} \|\rho\psi\|^2
  + \left(
  \nu - \frac{\delta}{32} (1+32/|I|^2)
  \right)
  \|\psi\|_{\sii(I\times\omega)}^2
  \,.
\end{equation*}
Employing the positivity of~$\nu$,
we choose
$
  \delta = \min\{1,32\nu/(1+32/|I|^2)\}
$
and thus obtain
\begin{equation*}
  H-\lambda_1
  \geq \frac{c'}{1+(x_1-x_1^0)^2}
\end{equation*}
with $c'=\delta/64$.

To conclude with~\eqref{Hardy.global}, we set
$$
  c := c' \min_{x_1 \in \Real} \frac{1+x_1^2}{1+(x_1-x_1^0)^2}
  \,,
$$
where the minimum is a positive constant depending on~$x_1^0$.
\end{proof}

In view of the unitary equivalence between~$H$ and $-\Delta_D^\Omega$
and since the longitudinal coordinate~$x_1$ is left invariant
by the rotation matrix in~\eqref{tube},
\eqref{Hardy.global}~is equivalent to the operator inequality
$$
  -\Delta_D^\Omega - \lambda_1
  \geq \frac{c}{1+x_1^2}
$$
in the form sense in $\sii(\Omega)$,
with the same constant~$c$.
This establishes Theorem~\ref{Thm.intro}
as a consequence of Theorem~\ref{Thm.global}
by noticing that $|x_1| \leq |x|$,
where, with an abuse of notation,
$|x|$~stand for the magnitude of the radial vector
in $\Real^3 \supset \Omega$.

\appendix
\section{Absence of eigenvalues in a one-dimensional problem}\label{App}
%
In this appendix we give a proof of Proposition~\ref{Prop.1D}.
Redefining~$\beta$ and~$\eps$ in~\eqref{H.1D},
we may assume, without loss of generality,
that the constant~$C_\omega$ is equal to~$1$
and consider in $\sii(\Real)$ just the Schr\"odinger operator
$H_\alpha := -\Delta^\Real + V_\alpha$,
$\Dom(H_\alpha) = \Dom(-\Delta^\Real) = H^2(\Real)$,
with the potential
$$
  V_\alpha := (\alpha\eps+\beta)^2 - \beta^2
  = \alpha^2\eps^2+2\alpha\beta\eps
  \,.
$$
We make the hypothesis that $\eps:\Real \to \Real$
is a continuous function of compact support
and $\beta \in \Real$.
Since the sign of~$\beta$ is not restricted,
we may assume, again without loss of generality,
that the coupling parameter~$\alpha$ is positive.
We prove Proposition~\ref{Prop.1D} by considering
the operator~$H_\alpha$ in the limit as $\alpha \to \infty$.

The support of~$\eps$ is a closure of a countable union of bounded
open intervals~$I_n$. Let $H_\alpha^{I_n}$ be the operator
in $\sii(I_n)$ that acts as~$H_\alpha$ inside~$I_n$
and satisfies the Neumann boundary conditions at~$\partial I_n$.
That is, $H_\alpha^{I_n} = -\Delta_N^{I_n} + V_\alpha$,
where $-\Delta_N^{I_n}$ is the Neumann Laplacian in $\sii(I_n)$.
Clearly,
\begin{equation}\label{bracketing}
  \inf\sigma(H_\alpha) \geq \min\left\{
  0, \inf_n \inf\sigma(H_\alpha^{I_n})
  \right\}
  \,.
\end{equation}
Here $\inf\sigma(H_\alpha^{I_n})$ is just the lowest eigenvalue
of~$H_\alpha^{I_n}$, since the latter is an operator with compact resolvent.
Our strategy to prove Proposition~\ref{Prop.1D}
is to show that each $\inf\sigma(H_\alpha^{I_n})$ is positive
for all sufficiently large~$\alpha$.

We shall need the following auxiliary result.
\begin{Lemma}\label{Lem.1D}
Let $-\infty<a<a'<b'<b<+\infty$. Then
$$
  \big[-\Delta_N^{(a,b)} + \alpha \chi_{(a',b')}\big]^{-1}
  \xrightarrow[\alpha \to \infty]{\mathrm{s}}
  \big[-\Delta_{ND}^{(a,a')}\big]^{-1}
  \oplus 0 \oplus
  \big[-\Delta_{DN}^{(b',b)}\big]^{-1}
  \,,
$$
where~$D,N$ stand for the Dirichlet and Neumann boundary conditions
at the respective parts of the interval and the direct-sums are with
respect to the decomposition $\sii((a,b)) \simeq \sii((a,a')) \oplus
\sii((a',b')) \oplus \sii((b',b))$.
\end{Lemma}
\begin{proof}
Given $F \in \sii((a,b))$, let~$\psi_\alpha := [-\Delta_N^{(a,b)} +
\alpha \chi_{(a',b')}+1]^{-1} F$. It satisfies the weak formulation
of the resolvent equation
\begin{equation}\label{weak}
  \int_a^b \bar{v}'(x) \, \psi_\alpha'(x) \,dx
  + \alpha \int_{a'}^{b'} \bar{v}(x) \, \psi_\alpha(x) \,dx
  + \int_{a}^{b} \bar{v}(x) \, \psi_\alpha(x) \,dx
  = \int_{a}^{b} \bar{v}(x) \, F(x) \,dx
\end{equation}
for every $v \in H^1((a,b))$.
Choosing $v=\psi_\alpha$, we get
$$
  \int_a^b |\psi_\alpha'(x)|^2 \,dx
  + \alpha \int_{a'}^{b'} |\psi_\alpha(x)|^2 \,dx
  + \int_{a}^{b} |\psi_\alpha(x)|^2 \,dx
  = \int_{a}^{b} \bar{\psi}_\alpha(x) \, F(x) \,dx
  \,,
$$
from which we deduce
$$
  \|\psi_\alpha\|^2_{H^1((a,b))} + \alpha \, \|\psi_\alpha\|^2_{\sii((a',b'))}
  \leq \|\psi_\alpha\|_{\sii((a,b))} \|F\|_{\sii((a,b))}
  \,.
$$
It follows that $\{\psi_\alpha\}_{\alpha>0}$
is a bounded family in $H^1((a,b))$
and therefore precompact in the weak topology of this space.
Let~$\psi_\infty$ be a weak limit point as $\alpha\to\infty$.
That is, for an increasing sequence of positive numbers
$\{\alpha_k\}_{k \in \Nat}$
such that $\alpha_k \to \infty$ as $k \to \infty$,
$\{\psi_{\alpha_k}\}_{k\in\Nat}$ converges weakly in
$H^1((a,b))$ to~$\psi_\infty$.
In fact, we may assume that it converges strongly in $\sii((a,b))$,
because the embedding $H^1((a,b))
\hookrightarrow \sii((a,b))$ is compact.
Dividing by~$\alpha_k$, we also see that
$$
  \|\psi_{\alpha_k}\|^2_{\sii((a',b'))} \longrightarrow 0
$$
as $k \to \infty$. Consequently, $\psi_\infty=0$ on $[a',b']$
(recall the embedding $H^1((a,b)) \hookrightarrow C^0([a,b])$). In
particular, $\psi_\infty$~satisfies the Dirichlet boundary
conditions at~$a',b'$. Choosing in~\eqref{weak} a test function~$v$
that vanish on $[a',b']$, restricting to the
subsequence~$\psi_{\alpha_k}$ and taking the limit $k \to \infty$,
we get that~$\psi_\infty$ is a solution to the boundary-value
problem
$$
\left\{
\begin{aligned}
  -\psi_\infty''+\psi_\infty &= F
  & \mbox{in}\quad & (a,a')\cup(b',b)
  \,, \\
  \psi_\infty' &= 0
  & \mbox{at}\quad & a,b
  \,, \\
  \psi_\infty &= 0
  & \mbox{at}\quad & a',b'
  \,. \\
\end{aligned}
\right.
$$
We have thus verified $
  \psi_\infty =
  \big[-\Delta_{ND}^{(a,a')}\big]^{-1} F
  \oplus 0 \oplus
  \big[-\Delta_{DN}^{(b',b)}\big]^{-1} F
$. The same limit is obtained for \emph{any} weak limit point of
$\{\psi_\alpha\}_{\alpha>0}$. Consequently,
$\psi_\alpha$~converges strongly in $\sii((a,b))$ to~$\psi_\infty$
as $\alpha \to \infty$, which was to be proved.
\end{proof}

Since the resolvents of Lemma~\ref{Lem.1D} are compact, we get
convergence of eigenvalues, in particular:
\begin{Corollary}\label{Corol.1D}
$$
  \inf\sigma\big(-\Delta_N^{(a,b)} + \alpha \chi_{(a',b')}\big)
  \ \xrightarrow[\alpha \to +\infty]{} \
  \min\left\{
  \left(\frac{\pi}{2(a'-a)}\right)^2,\left(\frac{\pi}{2(b-b')}\right)^2
  \right\}
  \,.
$$
\end{Corollary}

Now, recalling~\eqref{bracketing},
fix~$n$ and consider the lowest eigenvalue
$\inf\sigma(H_\alpha^{I_n})$ of~$H_\alpha^{I_n}$. Let~$I_n'$ be any
open subinterval of~$I_n$ such that $\overline{I_n'} \subset I_n$.
Then there exist positive constants $c_n=c_n(\beta,\eps
\upharpoonright I_n')$ and $\alpha_n=\alpha_n(\beta,\eps
\upharpoonright I_n')$ such that $V_\alpha \geq c_n \, \alpha^2$ on
$I_n'$ for all $\alpha \geq \alpha_n$. At the same time, $V_\alpha
\geq -\beta^2$ on the whole interval~$I_n$. Consequently,
$$
  H_\alpha^{I_n} \geq
  -\Delta_N^{I_n} + c_n \, \alpha^2 \, \chi_{I_n'}
  -\beta^2 \, \chi_{I_n \setminus I_n'}
  \,,
$$
where on the right hand side there is an operator to which
Lemma~\ref{Lem.1D} and its Corollary~\ref{Corol.1D} apply. It
follows that, choosing~$I_n'$ in such a way that~$|I_n \setminus
I_n'|$ is sufficiently small, we get that $\inf\sigma(-\Delta_N^{I}
+ c_n \, \alpha^2 \, \chi_{I_n'})$ will be larger than $\beta^2$ for
all sufficiently large~$\alpha$, and thus
$\inf\sigma(H_\alpha^{I_n})$ positive.

In general, the problem is that the constants measuring the largeness
of~$\alpha$ depend on~$n$, so that we have no uniform control over
the infimum appearing in~\eqref{bracketing}. This problem of course
does not appear under the hypothesis of Proposition~\ref{Prop.1D}
that the support of~$\eps$ is just a
closure of a \emph{finite} union of open intervals
(recall that in general it is just a \emph{countable} union).
Therefore we get the desired result under this extra assumption.

\section{The Neumann bracketing lowers the spectrum too much}\label{App.Neumann}
%
For the straight tube (\ie\ $\beta=0=\eps$),
imposing an extra Neumann condition at $\{x_1=0\}$
does not change the spectrum of the Dirichlet Laplacian in~$\Omega$.
This fact is used in \cite{K6-with-erratum,KZ1}
to go from the positivity of~$H-E_1$
for a locally twisted tube (\ie\ $\beta=0$)
to the Hardy inequality~\eqref{Hardy.tube}.
In this appendix we demonstrate that
this trick does not seem to be useful
for periodically twisted tubes (\ie\ $\beta\not=0$)
investigated in this paper.
Indeed, by imposing the supplementary Neumann condition at $\{x_1=0\}$,
one creates a spectrum below~$\lambda_1$ in~\eqref{essential}.

Imposing the ``Neumann condition''
at $\{x_1=0\}$ of the tube~\eqref{tube} means
that one considers instead of~$H$
the operator~$H^N$ which is associated in
$\sii(\Real\times\omega)$ with the form~$h^N$ that acts as~$h$
in~\eqref{form} but has a larger domain
$$
  \Dom(h^N) := \Dom_- \oplus \Dom_+
  \,,
$$
where~$\Dom_+$ (respectively, $\Dom_-$) denotes the set of functions
from $H_0^1(\Real\times\omega)$ restricted to~$(0,\infty)\times\omega$
(respectively, $(-\infty,0)\times\omega$).
Obviously, $H \geq H^N$.
However, this comparison is not much useful
in view of the following result.
\begin{Proposition}
Let $\beta \not= 0$ and $\eps=0$. Then
$$
  \inf\sigma(H^N) < \lambda_1
  \,.
$$
\end{Proposition}
\begin{proof}
The proof is based on the variational idea to construct a test
function $\psi\in\Dom(h^N)$ such that
$$
  Q^N[\psi] := h^N[\psi] - \lambda_1 \|\psi\|^2 < 0
  \,.
$$
First, we check that~$Q^N$ can be made asymptotically vanishing for
a suitably chosen sequence of test functions. Let
$\varphi:\Real\to\Real$ be a smooth function such that
$\varphi(x_1)=1$ for $x_1 \in [-1,1]$ and $\varphi(x_1)=0$ for
$|x_1| \geq 2$. Given any natural number $n \geq 1$, set
$\varphi_n(x_1):=\varphi(x_1/n)$. Note that $\varphi_n\chi \in
\Dom(h) \subset \Dom(h^N)$. Then
$$
  Q^N[\varphi_n\chi] = \|\dot\varphi_n\chi\|^2
  - 2 (\dot\varphi_n\chi,\beta\varphi_n\partial_\tau\chi)
  = \|\dot\varphi_n\|_{\sii(\Real)}^2
  =  \frac{1}{n}  \|\dot\varphi\|_{\sii(\Real)}^2
  \xrightarrow[n\to\infty]{} 0
  \,.
$$
Here the second equality follows by an integration by parts and the
normalisation of~$\chi$. Next, we add a small perturbation
$$
  \psi(x) := \varphi_n(x_1)\chi(x') + \delta \, \phi(x)
  \,,
$$
where~$\delta$ is a small real parameter and $\phi \in \Dom(h^N)$
will be specified later. Choosing~$\phi$ real-valued, we can write
\begin{equation}\label{3terms}
  Q^N[\psi] = Q^N[\varphi_n\chi]
  + 2 \, \delta \, Q^N(\phi,\varphi_n\chi)
  + \delta^2 \, Q^N[\phi]
  \,.
\end{equation}
Now we specify
$$
  \phi(x) := \rho(x_1) \tau(x') \chi(x')
  \,, \qquad \mbox{where} \qquad
  \tau(x') := -\arctan\frac{x_3}{x_2}
$$
is the angular variable and~$\rho$ is a real-valued function
supported in $[-1,0]$ such that $\rho(0)\not=0$. Note that~$\phi$
belongs to~$\Dom(h^N)$, although it does not belong to~$\Dom(h)$.
Integrating by parts in transverse variables and employing the
eigenvalue equation~$\chi$ satisfies, it is easy to check
$$
  Q^N(\phi,\varphi_n\chi)
  = (\partial_1\phi,\dot\varphi_n\chi)
  - (\partial_1\phi,\beta\varphi_n\partial_\tau\chi)
  + (\beta\phi,\dot\varphi_n\partial_\tau\chi)
  \,.
$$
Since~$\rho$ is supported in the interval where~$\varphi_n=1$, the
first and last integrals on the right hand side equal zero.
Integrating by parts in the remaining integral and using the
normalisation of~$\chi$, we conclude with
$$
  Q^N(\phi,\varphi_n\chi)
  = - (\partial_1\phi,\beta\varphi_n\partial_\tau\chi)
  = -\beta \rho(0) \int_\omega \tau \chi \partial_\tau\chi
  =  \frac{\beta\rho(0)}{2}
  \,.
$$
To fix the sign of the result, let us choose  $\rho(0)=
- \beta$. Then the mixed term on the right hand side
of~\eqref{3terms} is negative and independent of~$n$ for all
positive~$\delta$. Moreover, the sum with the last term can be
guaranteed to remain negative by choosing~$\delta$ sufficiently
small. Finally, we choose~$n$ sufficiently large in order to make
the sum of all terms in~\eqref{3terms} negative.
\end{proof}
%

\subsection*{Acknowledgment}
%
The research of D.K.\ was partially supported
by the project RVO61389005 and the GACR grant No.\ 14-06818S.
The author also acknowledges the award from 
the \emph{Neuron fund for support of science},
Czech Republic.

%
\bibliography{bib}
\bibliographystyle{amsplain}

\end{document}